\documentclass[12pt,a4paper]{article}
\usepackage[utf8x]{inputenc}
\usepackage{amsmath}
\usepackage{amsfonts}
\usepackage{amssymb}
\usepackage{amsthm}
\usepackage{mathrsfs}
\usepackage{fancyhdr}
\usepackage{graphicx}
\usepackage[usenames,x11names]{xcolor}
\usepackage{hyperref}
\usepackage{framed}
\usepackage[top=2cm,bottom=2cm,marginparwidth=4cm]{geometry}
\usepackage{cancel}
\usepackage{array}   
\usepackage{bigints}

\title{On some products taken over the prime numbers}

\author{\textsc{Abdelmalek BEDHOUCHE} and \textsc{Bakir FARHI}\thanks{Correspending author.} \\
Laboratoire de Math\'ematiques appliqu\'ees \\
Facult\'e des Sciences Exactes \\
Universit\'e de Bejaia, 06000 Bejaia, Algeria \\[1mm]
\href{mailto:abdelmalek.bedhouche@univ-bejaia.dz}{abdelmalek.bedhouche@univ-bejaia.dz} (A. Bedhouche), \\
\href{bakir.farhi@univ-bejaia.dz}{bakir.farhi@univ-bejaia.dz} (B. Farhi)
}

\date{}

\let\up=\textsuperscript

\def\R{{\mathbb R}}

\def\N{{\mathbb N}}
\def\Z{{\mathbb Z}}

\def\lcm{\mathrm{lcm}}
\def\c{\mathrm{c}}

\newcommand{\card}[1]{\mathrm{Card}\,#1}

\def\EMdash{\leavevmode\hbox to 10.6mm{\vrule height .63ex depth -.59ex
    width 10mm\hfill}}

\theoremstyle{plain}
\numberwithin{equation}{section}
\newtheorem{thm}{Theorem}[section]

\newtheorem{lemma}[thm]{Lemma}

\newtheorem{prop}[thm]{Proposition}
\newtheorem{coll}[thm]{Corollary}
\newtheorem{conj}[thm]{Conjecture}

\theoremstyle{definition}

\theoremstyle{remark}

\newtheorem{rmks}[thm]{Remarks}

\begin{document}
\maketitle

\begin{abstract}
This paper is devoted to study some expressions of the type $\prod_{p} p^{\lfloor\frac{x}{f(p)}\rfloor}$, where $x$ is a nonnegative real number, $f$ is an arithmetic function satisfying some conditions, and the product is over the primes $p$. We begin by proving that such expressions can be expressed by using the $\mathrm{lcm}$ function, without any reference to prime numbers; we illustrate this result with several examples. The rest of the paper is devoted to study the two particular cases related to $f(m) = m$ and $f(m) = m - 1$. In both cases, we found arithmetic properties and analytic estimates for the underlying expressions. We also put forward an important conjecture for the case $f(m) = m - 1$, which depends on the counting of the prime numbers of a special form. 
\end{abstract}

\noindent\textbf{MSC 2010:} Primary 11A05, 11A41, 11N37. \\
\textbf{Keywords:} Prime factorisation, least common multiple, prime numbers of a special form, asymptotic estimates of arithmetic functions.

\section{Introduction and Notation}

Throughout this paper, we let $\N^*$ denote the set of positive integers and $\mathscr{P}$ the set of prime numbers. We let $\card \mathscr{A}$ denote the cardinal of a given finite set $\mathscr{A}$. For a given prime number $p$, we let $\vartheta_p$ denote the usual $p$-adic valuation. For $x \in \R$, we let $\lfloor x\rfloor$ denote the integer-part of $x$. For $N , b \in \N$, with $b \geq 2$, the expansion of $N$ in base $b$ is denoted by $N = \overline{a_k a_{k - 1} \dots a_1 a_0}_{(b)}$, meaning that $N = a_0 + b a_1 + b^2 a_2 + \dots + b^k a_k$ (with $k \in \N$, $a_0 , a_1 , \dots , a_k \in \{0 , 1 , \dots , b - 1\}$ and $a_k \neq 0$). In such a context, we let $S_b(N)$ denote the sum of base-$b$ digits of $N$, that is $S_b(N) := a_0 + a_1 + \dots + a_k$. Further, we let $\pi$ and $\theta$ respectively denote the prime-counting function and the Chebyshev theta function, defined by:
$$
\pi(x) := \sum_{\begin{subarray}{c}
p \text{ prime} \\
p \leq x
\end{subarray}} 1 ~~~~\text{and}~~~~
\theta(x) := \sum_{\begin{subarray}{c}
p \text{ prime} \\
p \leq x
\end{subarray}} \log{p} ~~~~~~~~~~ (\forall x \in \R^+) .
$$
The prime number theorem states that $\pi(x) \sim_{+ \infty} \frac{x}{\log{x}}$. Other equivalent statements are: $\theta(x) \sim_{+ \infty} x$ and $\log\lcm(1 , 2 , \dots , n) \sim_{+ \infty} n$ (see e.g., \cite[Chapter 4]{kon}). The weaker estimates $\pi(x) = O\left(\frac{x}{\log{x}}\right)$, $\theta(x) = O(x)$ and $\log\lcm(1 , 2 , \dots , n) = O(n)$ are called Chebyshev's estimates. In section \ref{sec4}, we use extensively Landau's big $O$ notation which we sometimes specify as follows: if $f$ and $g$ are two real functions, with $g > 0$, defined on some interval $I$ of $\R$, and depending on a parameter $t$, then we write $f = O_{\perp t}(g)$ if there exists a positive constant $M$, \underline{not depending on $t$}, such that $\vert f(x)\vert \leq M g(x)$ ($\forall x \in I$).

In number theory, it is common that a prime factorisation of some special numbers $N$ makes appear, as exponents of each prime $p$, expressions of the form $\lfloor\frac{u_N}{f(p)}\rfloor$ or a sum of such expressions. The most famous example is perhaps the Legendre formula stating that for any natural number $n$, we have
\begin{equation}\label{eq24}
n! = \prod_{p \text{ prime}} p^{\lfloor\frac{n}{p}\rfloor + \lfloor\frac{n}{p^2}\rfloor + \lfloor\frac{n}{p^3}\rfloor + \dots} ,
\end{equation}
which may be also reformulate in terms of base expansions as follows:
\begin{equation}\label{eq25}
n! = \prod_{p \text{ prime}} p^{\frac{n - S_p(n)}{p - 1}} .
\end{equation}
(See e.g., \cite[pages 76-77]{moll}). Another famous example is the formula of the least common multiple of the first consecutive positive integers:
\begin{equation}\label{eq26}
\lcm(1 , 2 , \dots , n) = \prod_{p \text{ prime}} p^{\lfloor\frac{\log{n}}{\log{p}}\rfloor} ~~~~~~~~~~ (\forall n \in \N) .
\end{equation}
Among other examples which are less known, we can cite the following
\begin{equation}\label{eq27}
\lcm\Big\{i_1 i_2 \cdots i_k ~;~ k \in \N , i_1 , i_2 , \dots , i_k \in \N^* , i_1 + i_2 + \dots + i_k \leq n\Big\} = \prod_{p \text{ prime}} p^{\lfloor\frac{n}{p}\rfloor} ,
\end{equation}
which is pointed out in the book of Cahen and Chabert \cite[page 246]{cah} and also by Farhi \cite{far} in the context of the integer-valued polynomials. Basing on the remark that in both Formulas \eqref{eq24}, \eqref{eq26} and \eqref{eq27}, the right-hand side (which is a product taken over the primes) is interpreted without any reference to prime numbers, we may naturally ask if an expression of a general type $\prod_{p \text{ prime}} p^{\lfloor\frac{x}{f_1(p)}\rfloor + \lfloor\frac{x}{f_2(p)}\rfloor + \dots}$ (where $x \in \R^+$ and ${(f_i)}_i$ is a sequence of positive functions, satisfying some regularity conditions) possess the same property; that is, it has an interpretation without reference to the primes. In this paper, we study only the case of the products 
$$
\pi_f(x) := \prod_{p \text{ prime}} p^{\left\lfloor\frac{x}{f(p)}\right\rfloor} ,
$$
for which we answer affirmatively to the previous question (under some hypothesis on $f$). After giving several applications of our result, we focus our study on the two particular cases $f(p) = p$ and $f(p) = p - 1$. Because in both cases, there is no loss of generality to take $x$ an integer, we are led to define for any $n \in \N$:
$$
\rho_n := \prod_{p \text{ prime}} p^{\left\lfloor\frac{n}{p}\right\rfloor} ~~~~\text{and}~~~~ \sigma_n := \prod_{p \text{ prime}} p^{\left\lfloor\frac{n}{p - 1}\right\rfloor} .
$$
We begin with the arithmetic study of $\rho_n$ and $\sigma_n$ by establishing several arithmetic properties concerning them; especially, we obtain for $\sigma_n$ a nontrivial divisor and a nontivial multiple. Moreover, we determine the $p$-adic valuations of the integers $\frac{\sigma_n}{n!}$ when the prime $p$ is large enough compared to $\sqrt{n}$; we discover that the prime numbers of the form $\lfloor\frac{n}{k} + 1\rfloor$ ($k \in \N^*$, $k < \sqrt{n + 1} + 1$) play a vital role in the arithmetic nature of the $\sigma_n$'s (this phenomenon will pop up again when studying analytically the $\sigma_n$'s). In another direction, we find asymptotic estimates for $\log{\rho_n}$ and $\log{\sigma_n}$. However, due to the difficulties encountered in counting the prime numbers of the form $\lfloor\frac{n}{k} + 1\rfloor$ ($1 \leq k \leq \sqrt{n}$), the optimal estimate of $\log{\sigma_n}$ is only given conjecturally by leaning on heuristic reasoning. We finally conclude the paper by pointing out the connection between our arithmetic and analytic studies concerning the numbers $\sigma_n$. 

\section{An expression of $\pi_f$ using the $\lcm$'s}
Our stronger result of expressing $\pi_f$ in terms of the $\lcm$'s (without any reference to prime numbers) is the following:
\begin{thm}\label{t1}
Let $f : \N^* \rightarrow \R_+$ be an arithmetic function such that $f(\N^* \setminus \{1\}) \subset \R^*_+$ (i.e., $f$ does not vanish except at $1$ eventually). Consider the set $\N^* \setminus \{1\}$ equipped with the partial order relation ``divide'' and the set $\R^*_+$ equipped with the usual total order relation ``$\leq$'', and suppose that the map:
$$
\begin{array}{rcl}
\widetilde{f} :~ \N^* \setminus \{1\} & \longrightarrow & \R^*_+ \\[2mm]
n & \longmapsto & \dfrac{f(n)}{\log{n}}
\end{array}
$$
is nondecreasing with respect to these two orders. Then, we have for any $x \in \R^+$:
\begin{multline*}
\prod_{p \text{ prime}} p^{\left\lfloor\frac{x}{f(p)}\right\rfloor} = \lcm\Big\{i_1 i_2 \cdots i_k ~;~ k \in \N , i_1 , i_2 , \dots , i_k \in \N^* , \\[-5mm]
f(i_1) + f(i_2) + \dots + f(i_k) \leq x\Big\} .
\end{multline*}
\end{thm}

In order to present a clean proof of Theorem \ref{t1}, we go through the following lemma:

\begin{lemma}\label{l1}
Let $f : \N^* \rightarrow \R_+$ as in Theorem \ref{t1}. Then, for any prime number $p$ and any positive integer $a$, we have
$$
\vartheta_p(a) \leq \frac{f(a)}{f(p)} .
$$
\end{lemma}

\begin{proof}
Let $p$ be a prime number and $a$ be a positive integer. Since the inequality of the lemma is trivial when $\vartheta_p(a) = 0$, we may suppose that $\vartheta_p(a) \geq 1$; that is $p \mid a$. Setting $\alpha = \vartheta_p(a)$, we can write $a = b p^{\alpha}$ for some $b \in \N^*$ with $p \nmid b$. Thus
\begin{equation}\label{eq1}
\alpha \leq \alpha + \frac{\log{b}}{\log{p}} = \frac{\log(b p^{\alpha})}{\log{p}} = \frac{\log{a}}{\log{p}} .
\end{equation}
Next, the fact that $p \mid a$ implies (according to our assumptions on $f$) that:
$$
\frac{f(p)}{\log{p}} \leq \frac{f(a)}{\log{a}} ;
$$
that is
\begin{equation}\label{eq2}
\frac{\log{a}}{\log{p}} \leq \frac{f(a)}{f(p)} . 
\end{equation}
Combining \eqref{eq1} and \eqref{eq2}, we get
$$
\alpha \leq \frac{f(a)}{f(p)} ,
$$
as required. This completes the proof of the lemma.
\end{proof}

\begin{proof}[Proof of Theorem \ref{t1}]
Let $x \in \R_+$ be fixed. For a given prime number $p$, the $p$-adic valuation of the left-hand side of the identity of the theorem is equal to $\lfloor\frac{x}{f(p)}\rfloor$, while the $p$-adic valuation of the right-hand side of the same identity is equal to $\ell_p := \max\{\vartheta_p(i_1 i_2 \dots i_k) ; k \in \N , i_1 , \dots , i_k \in \N^* , f(i_1) + \dots + f(i_k) \leq x\}$. So, we have to show that $\ell_p = \lfloor\frac{x}{f(p)}\rfloor$ (for any prime number $p$). To do so, we are going to prove the two inequalities $\ell_p \geq \lfloor\frac{x}{f(p)}\rfloor$ and $\ell_p \leq \lfloor\frac{x}{f(p)}\rfloor$ (where $p$ is a given prime number).

First, for a given prime number $p$, let us show that $\ell_p \geq \lfloor\frac{x}{f(p)}\rfloor$. By considering the particular natural number:
$$
k = \left\lfloor\frac{x}{f(p)}\right\rfloor
$$
and the particular positive integers:
$$
i_1 = i_2 = \dots = i_k = p ,
$$
we get
$$
f(i_1) + f(i_2) + \dots + f(i_k) = k f(p) = \left\lfloor\frac{x}{f(p)}\right\rfloor f(p) \leq x .
$$
Thus (according to the definition of $\ell_p$):
$$
\ell_p \geq \vartheta_p\left(i_1 i_2 \cdots i_k\right) = \vartheta_p\left(p^k\right) = k = \left\lfloor\frac{x}{f(p)}\right\rfloor ,
$$
as required.

Now, for a given prime number $p$, let us show that $\ell_p \leq \lfloor\frac{x}{f(p)}\rfloor$. For any $k \in \N$ and any $i_1 , i_2 , \dots , i_k \in \N^*$, with $f(i_1) + f(i_2) + \dots + f(i_k) \leq x$, we have
\begin{align*}
\vartheta_p\left(i_1 i_2 \cdots i_k\right) & = \vartheta_p(i_1) + \vartheta_p(i_2) + \dots + \vartheta_p(i_k) \\
& \leq \frac{f(i_1)}{f(p)} + \frac{f(i_2)}{f(p)} + \dots + \frac{f(i_k)}{f(p)} ~~~~~~~~~~ (\text{according to Lemma \ref{l1}}) \\
& = \frac{f(i_1) + f(i_2) + \dots + f(i_k)}{f(p)} \\
& \leq \frac{x}{f(p)} ;
\end{align*}
but since $\vartheta_p(i_1 i_2 \cdots i_k) \in \N$, it follows that:
$$
\vartheta_p\left(i_1 i_2 \cdots i_k\right) \leq \left\lfloor\frac{x}{f(p)}\right\rfloor .
$$
The definition of $\ell_p$ concludes that
$$
\ell_p \leq \left\lfloor\frac{x}{f(p)}\right\rfloor ,
$$
as required. This completes the proof.
\end{proof}

\begin{rmks}\label{rmks1}
Let us put ourselves in the situation of Theorem \ref{t1}.
\begin{enumerate}
\item If the map $\widetilde{f}$ is nondecreasing in the usual sense (i.e., with respect to the usual orders of the two sets $\N^* \setminus \{1\}$ and $\R_+^*$) then it remains nondecresing in the sense imposed by Theorem \ref{t1} (this immediately follows from the implication: $a \mid b \Rightarrow a \leq b$, $\forall a , b \in \N^*$).
\item More generally than the previous item, if the restriction of the map $\widetilde{f}$ on $\N^* \setminus \{1 , 2\}$ is nondecreasing in the usual sense and $\widetilde{f}(2) \leq \widetilde{f}(4)$ then $\widetilde{f}$ is nondecreasing in the sense imposed by Theorem \ref{t1}.
\end{enumerate}
\end{rmks}

Now, from Theorem \ref{t1}, we derive the following corollary in which the condition imposed on $f$ is made simpler.

\begin{coll}\label{coll1}
Let $f : \N^* \rightarrow \R_+$ be an arithmetic function satisfying $f(\N^* \setminus \{1\}) \subset \R_+^*$. Suppose that the map
$$
\begin{array}{rcl}
\N^* \setminus \{1\} & \longrightarrow & \R_+^* \\[2mm]
n & \longmapsto & \dfrac{f(n)}{n}
\end{array}
$$
is nondecreasing in the usual sense (i.e., with respect to the usual order of $\R$). Then we have for any $x \in \R_+$: 
\begin{multline*}
\prod_{p \text{ prime}} p^{\left\lfloor\frac{x}{f(p)}\right\rfloor} = \lcm\Big\{i_1 i_2 \cdots i_k ~;~ k \in \N , i_1 , i_2 , \dots , i_k \in \N^* , \\[-5mm]
f(i_1) + f(i_2) + \dots + f(i_k) \leq x\Big\} .
\end{multline*}
\end{coll}

\begin{proof}
We use Theorem \ref{t1} together with Item 2 of Remarks \ref{rmks1}. We remark that $\widetilde{f}$ (defined as in Theorem \ref{t1}) is the product of the two functions: $n \mapsto \frac{f(n)}{n}$ (supposed nondecreasing in the usual sense on $\N^* \setminus \{1\}$) and $n \mapsto \frac{n}{\log{n}}$ (which is nondecreasing on $\N^* \setminus \{1 , 2\} = \{3 , 4 , 5 , \dots\}$, as a simple study of function shows). So, $\widetilde{f}$ is nondecreasing on $\N^* \setminus \{1 , 2\}$ in the usual sense. In addition, we have
$$
\widetilde{f}(2) = \frac{f(2)}{\log{2}} = \frac{f(2)}{2} \cdot \frac{2}{\log{2}} = \frac{f(2)}{2} \cdot \frac{4}{\log{4}} \leq \frac{f(4)}{4} \cdot \frac{4}{\log{4}}
$$
(since $n \mapsto \frac{f(n)}{n}$ is supposed nondecreasing in the usual sense on $\N^* \setminus \{1\}$). That is
$$
\widetilde{f}(2) \leq \frac{f(4)}{\log{4}} = \widetilde{f}(4) .
$$
The conclusion follows from Item 2 of Remarks \ref{rmks1} and Theorem \ref{t1}.
\end{proof}

\subsection*{Some applications:}
{\large\bf 1.} By applying Theorem \ref{t1} for $f(m) = \log{m}$ (which cleary satisfies the required conditions), we obtain that for any $x \in \R_+$, we have
\begin{align*}
\prod_{p \text{ prime}} p^{\left\lfloor\frac{x}{\log{p}}\right\rfloor} & = \lcm\Big\{i_1 i_2 \cdots i_k ~;~ k \in \N , i_1 , i_2 , \dots , i_k \in \N^* , \\[-5mm]
& \hspace*{6cm} \log{i_1} + \log{i_2} + \dots + \log{i_k} \leq x\Big\} \\
& = \lcm\Big\{i_1 i_2 \cdots i_k ~;~ k \in \N , i_1 , i_2 , \dots , i_k \in \N^* , i_1 i_2 \cdots i_k \leq e^x\Big\} \\
& = \lcm\Big\{1 , 2 , \dots , \left\lfloor e^x\right\rfloor\Big\} .
\end{align*}
By taking in particular $x = \log{n}$ ($n \in \N^*$), we obtain the well-known formula:
$$
\prod_{p \text{ prime}} p^{\left\lfloor\frac{\log{n}}{\log{p}}\right\rfloor} = \lcm\left\{1 , 2 , \dots , n\right\} ~~~~~~~~~~ (\forall n \in \N^*) .
$$

\noindent{\large\bf 2.} By applying Corollary \ref{coll1} for the function $f(m) = m$ (which clearly satisfies the imposed conditions), we obtain in particular that for all $n \in \N$, we have
\begin{multline}\label{eq3}
\prod_{p \text{ prime}} p^{\left\lfloor\frac{n}{p}\right\rfloor} = \lcm\Big\{i_1 i_2 \cdots i_k ~;~ k \in \N , i_1 , i_2 , \dots , i_k \in \N^* , \\[-5mm]
i_1 + i_2 + \dots + i_k \leq n\Big\} ,
\end{multline}
which is already pointed out by Cahen and Chabert \cite{cah} and by Farhi \cite{far}.

\medskip

\noindent{\large\bf 3.} (Generalization of \eqref{eq3}). Let $\alpha \geq 1$. By applying Corollary \ref{coll1} for the function $f(m) = m^{\alpha}$ (which clearly satisfies the imposed conditions), we obtain in particular that for all $n \in \N$, we have
\begin{multline*}\label{eq4}
\prod_{p \text{ prime}} p^{\left\lfloor\frac{n}{p^{\alpha}}\right\rfloor} = \lcm\Big\{i_1 i_2 \cdots i_k ~;~ k \in \N , i_1 , i_2 , \dots , i_k \in \N^* , \\[-5mm]
i_1^{\alpha} + i_2^{\alpha} + \dots + i_k^{\alpha} \leq n\Big\} .
\end{multline*}

\noindent{\large\bf 4.} For all $n , k \in \N$, with $n \geq k$, let us define (as in \cite{far}):
$$
q_{n , k} := \lcm\Big\{i_1 i_2 \cdots i_k ~;~ i_1 , i_2 , \dots , i_k \in \N^* , i_1 + i_2 + \dots + i_k \leq n\Big\} .
$$
Note that these numbers have been already encountered and studied by Farhi \cite{far} in a context relating to the integer-valued polynomials. By applying Corollary \ref{coll1} for the function $f(m) = m - 1$ (which clearly satisfies the imposed conditions), we obtain that for all $n \in \N$, we have 
\begin{eqnarray}
\prod_{p \text{ prime}} p^{\left\lfloor\frac{n}{p - 1}\right\rfloor} & = & \lcm\Big\{i_1 i_2 \cdots i_k ~;~ k \in \N , i_1 , i_2 , \dots , i_k \in \N^* , \notag \\[-5mm]
& ~ & \hspace*{3.5cm} (i_1 - 1) + (i_2 - 1) + \dots + (i_k - 1) \leq n\Big\}  \notag \\
& \hspace*{-2.4cm} = & \hspace*{-1.2cm} \lcm\Big\{i_1 i_2 \cdots i_k ~;~ k \in \N , i_1 , i_2 , \dots , i_k \in \N^* , i_1 + i_2 + \dots + i_k \leq n + k\Big\} \notag \\
& \hspace*{-2.4cm} = & \hspace*{-1.2cm} \lcm\Big\{q_{n + k , k} ~;~ k \in \N \Big\} , \label{eq5} 
\end{eqnarray}
which remarkably represents the least common multiple of the $n$\up{th} diagonal of the arithmetic triangle of the $q_{i , j}$'s, beginning as follows (see \cite{far}):

\begin{table}[!h]
$$
\begin{array}{llllllll}
1 & ~ & ~ & ~ & ~ & ~ & ~ & \\
1 & 1 & ~ & ~ & ~ & ~ & ~ & \\
1 & 2 & 1 & ~ & ~ & ~ & ~ & \\
1 & 6 & 2 & 1 & ~ & ~ & ~ & \\
1 & 12 & 12 & 2 & 1 & ~ & ~ & \\
1 & 60 & 12 & 12 & 2 & 1 & ~ & \\
1 & 60 & 360 & 24 & 12 & 2 & 1 & \\
1 & 420 & 360 & 360 & 24 & 12 & 2 & 1
\end{array}
$$
\caption{The triangle of the $q_{n , k}$'s for $0 \leq k \leq n \leq 7$}
\end{table}

For a given $n \in \N$, let $D_n = {(d_{n , k})}_{k \in \N}$ denote the sequence of the $n$\up{th} diagonal of the above triangle, that is
\begin{eqnarray}
d_{n , k} & := & q_{n + k , k} \notag \\
& = & \lcm\Big\{i_1 i_2 \cdots i_k ~;~ i_1 , i_2 , \dots , i_k \in \N^* , i_1 + i_2 + \dots + i_k \leq n + k\Big\}  \label{eq6}
\end{eqnarray}
($\forall k \in \N$). In order to simplify Formula \eqref{eq5}, we are going to show that the sequences $D_n$ ($n \in \N$) are all nondecreasing in the divisibility sense and eventually constant. Precisely, we have the following proposition:
\begin{prop}\label{p1}
For all $n , k \in \N$, we have
$$
d_{n , k} ~~\text{divides}~~ d_{n , k + 1} .
$$
If in addition $k \geq n$, then we have
$$
d_{n , k} = d_{n , n} .
$$
\end{prop}
\begin{proof}
Let $n , k \in \N$ be fixed and let $i_1 , i_2 , \dots , i_k \in \N^*$ such that $i_1 + i_2 + \dots + i_k \leq n + k$. By setting $i_{k + 1} = 1$, we have $i_1 + i_2 + \dots + i_k + i_{k + 1} \leq n + k + 1$; thus (by \eqref{eq6}) $d_{n , k + 1}$ is a multiple of $i_1 i_2 \cdots i_k i_{k + 1} = i_1 i_2 \cdots i_k$. Since this holds for any $i_1 , i_2 , \dots , i_k \in \N^*$ such that $i_1 + i_2 + \dots + i_k \leq n + k$, we derive that $d_{n , k + 1}$ is a multiple of $d_{n , k}$, as required.

Now, let us prove the second part of the proposition. So, suppose that $k \geq n$ and let us prove that $d_{n , k} = d_{n , n}$. It follows from an immediate induction leaning on the result of the first part of the proposition (proved above) that $d_{n , n} \mid d_{n , k}$. So, it remains to prove that $d_{n , k} \mid d_{n , n}$. Let $i_1 , i_2 , \dots , i_k \in \N^*$ such that $i_1 + i_2 + \dots + i_k \leq n + k$. Let $\ell \in \N$ denote the number of indices $i_r$ ($1 \leq r \leq k$) which are equal to $1$; so we have exactly $(k - \ell)$ indices $i_r$ which are $\geq 2$. Thus we have
$$
i_1 + i_2 + \dots + i_k \geq \ell + 2 (k - \ell) = 2 k - \ell .
$$
But since $i_1 + i_2 + \dots + i_k \leq n + k$, we derive that $2 k - \ell \leq n + k$, which gives $\ell \geq k - n$. This proves that we have at least $(k - n)$ indices $i_r$ which are equal to $1$. By assuming (without loss of generality) that those indices are $i_{n + 1} , i_{n + 2} , \dots , i_k$ (i.e., $i_{n + 1} = i_{n + 2} = \dots = i_k = 1$), we get
\begin{align*}
i_1 i_2 \cdots i_n & = i_1 i_2 \cdots i_k \\[-5mm]
\intertext{and} \\[-1cm]
i_1 + i_2 + \dots + i_n & = \left(i_1 + i_2 + \dots + i_k\right) - (k - n) \\
& \leq (n + k) - (k - n) \\
& = 2 n .
\end{align*}
This shows that each product $i_1 i_2 \cdots i_k$ occurring in the definition of $d_{n , k}$ reduces (by permuting the $i_r$'s and eliminate those of them which are equal to $1$) to a product $j_1 j_2 \cdots j_n$ which occurs in the definition of $d_{n , n}$. Consequently $d_{n , k} \mid d_{n , n}$, as required. This completes the proof of the proposition.
\end{proof}

Using Proposition \ref{p1}, we have for any $n \in \N$:
\begin{align*}
\lcm\Big\{q_{n + k , k} ~;~ k \in \N\Big\} & = \lcm\Big\{d_{n , k} ~;~ k \in \N\Big\} \\
& = d_{n , n} \\
& \hspace*{-5mm} = \lcm\Big\{i_1 i_2 \cdots i_n ~;~ i_1 , i_2 , \dots , i_n \in \N^* , i_1 + i_2 + \dots + i_n \leq 2 n\Big\} .
\end{align*}

\noindent This concludes to the following interesting corollary, simplifying Formula \eqref{eq5}:

\begin{coll}\label{coll2}
For any $n \in \N$, we have
\begin{equation}
\prod_{p \text{ prime}} p^{\left\lfloor\frac{n}{p - 1}\right\rfloor} = \lcm\Big\{i_1 i_2 \cdots i_n ~;~ i_1 , i_2 , \dots , i_n \in \N^* , i_1 + i_2 + \dots + i_n \leq 2 n\Big\} . \tag*{$\square$}
\end{equation}
\end{coll}

\section[Arithmetic results]{Arithmetic results on the numbers $\rho_n$ and $\sigma_n$}\label{sec3}
A certain number of arithmetic properties concerning the numbers $\rho_n$ and $\sigma_n$ are either immediate or quite easy to prove. We have gathered them in the following proposition:

\begin{prop}\label{p2}
For any natural number $n$, we have
\begin{enumerate}
\item[{\rm (i)}] $\rho_n \mid \rho_{n + 1}$, $\sigma_n \mid \sigma_{n + 1}$, and $\rho_n \mid \sigma_n$;
\item[{\rm (ii)}] $\rho_n \mid n!$;
\item[{\rm (iii)}] $n! \mid \sigma_n$ and $\sigma_n \mid (2 n)!$;
\item[{\rm (iv)}] $\sigma_{2 n + 1} = 2 \sigma_{2 n}$.
\end{enumerate}
\end{prop} 

\begin{proof}
Let $n \in \N$ be fixed. The properties of Item (i) are trivial. The property of Item (ii) immediately follows from the Legendre formula providing the decomposition of $n!$ into a product of prime factors. For Item (iii), the fact that $n! \mid \sigma_n$ follows from the inequality:
$$
\frac{n}{p - 1} = \frac{n}{p} + \frac{n}{p^2} + \frac{n}{p^3} + \dots \geq \left\lfloor\frac{n}{p}\right\rfloor + \left\lfloor\frac{n}{p^2}\right\rfloor + \left\lfloor\frac{n}{p^3}\right\rfloor + \dots
$$
together with the Legendre formula. Next, to prove that $\sigma_n \mid (2 n)!$, we use Corollary \ref{coll2}. For any $i_1 , i_2 , \dots , i_n \in \N^*$ satisfying $i_1 + i_2 + \dots + i_n \leq 2 n$, we have that $i_1 i_2 \cdots i_n \mid i_1! i_2! \cdots i_n! \mid (i_1 + i_2 + \dots + i_n)! \mid (2 n)!$. Thus $\lcm\{i_1 i_2 \cdots i_n ~;~ i_1 , i_2 , \dots , i_n$ $\in \N^* , i_1 + i_2 + \dots + i_n \leq 2 n\} \mid (2 n)!$; that is (according to Corollary \ref{coll2}): $\sigma_n \mid (2 n)!$. Let us finally prove Item (iv). First, we have $\vartheta_2(\sigma_{2 n + 1}) = \lfloor\frac{2 n + 1}{2 - 1}\rfloor = 2 n + 1$ and $\vartheta_2(2 \sigma_{2 n}) = 1 + \vartheta_2(\sigma_{2 n}) = 1 + \lfloor\frac{2 n}{2 - 1}\rfloor = 2 n + 1$; hence $\vartheta_2(\sigma_{2 n + 1}) = \vartheta_2(2 \sigma_{2 n})$. Next, for any odd prime $p$, since the odd number $(2 n + 1)$ cannot be a multiple of the even number $(p - 1)$ then we have
$$
\left\lfloor\frac{2 n + 1}{p - 1}\right\rfloor = \left\lfloor\frac{2 n}{p - 1}\right\rfloor ;
$$
that is
$$
\vartheta_p(\sigma_{2 n + 1}) = \vartheta_p(\sigma_{2 n}) = \vartheta_p(2 \sigma_{2 n}) .
$$
Consequently, we have $\vartheta_q(\sigma_{2 n + 1}) = \vartheta_q(2 \sigma_{2 n})$ (for any prime number $q$), concluding that $\sigma_{2 n + 1} = 2 \sigma_{2 n}$, as required. This completes the proof of the proposition.
\end{proof}

In the following proposition, we shall improve Item (iii) of Proposition \ref{p2}. It appears that this improvement is optimal (understanding that we use uniquely simple expressions).

\begin{prop}\label{p3}
For any natural number $n$, we have
$$
(n + 1)! ~\mid~ \sigma_n ~~\text{and}~~ \sigma_n ~\mid~ n! \, \lcm(1 , 2 , \dots , n , n + 1) .
$$
\end{prop}

\begin{proof}
Let $n \in \N$ be fixed. We have to show that for any prime $p$, we have
\begin{equation}\label{eq7}
\vartheta_p\left((n + 1)!\right) \leq \vartheta_p\left(\sigma_n\right) \leq \vartheta_p\left(n! \, \lcm(1 , 2 , \dots , n , n + 1)\right) .
\end{equation}
Let $p$ be a fixed prime number and let us prove \eqref{eq7}. By setting $e$ the largest prime number satisfying $p^e \leq n + 1$, we have that $\vartheta_p(n!) = \sum_{i = 1}^{e} \left\lfloor\frac{n}{p^i}\right\rfloor$, $\vartheta_p((n + 1)!) = \sum_{i = 1}^{e} \left\lfloor\frac{n + 1}{p^i}\right\rfloor$ (according to the Legendre formula), $\vartheta_p(\sigma_n) = \left\lfloor\frac{n}{p - 1}\right\rfloor$ (by definition of $\sigma_n$), and $\vartheta_p\left(\lcm(1 , 2 , \dots , n + 1)\right) = e$. So \eqref{eq7} reduces to
\begin{equation}\label{eq8}
\sum_{i = 1}^{e} \left\lfloor\frac{n + 1}{p^i}\right\rfloor \leq \left\lfloor\frac{n}{p - 1}\right\rfloor \leq \sum_{i = 1}^{e} \left\lfloor\frac{n}{p^i}\right\rfloor + e .
\end{equation}
On the one hand, we have
$$
\sum_{i = 1}^{e} \left\lfloor\frac{n + 1}{p^i}\right\rfloor \leq \sum_{i = 1}^{e} \frac{n + 1}{p^i} = \frac{n + 1}{p - 1} \left(1 - \frac{1}{p^e}\right) \leq \frac{n}{p - 1}
$$
(since $p^e \leq n + 1$). But since $\sum_{i = 1}^{e} \left\lfloor\frac{n + 1}{p^i}\right\rfloor$ is an integer, we derive that
$$
\sum_{i = 1}^{e} \left\lfloor\frac{n + 1}{p^i}\right\rfloor \leq \left\lfloor\frac{n}{p - 1}\right\rfloor ,
$$
confirming the left inequality in \eqref{eq8}. On the other hand, by leaning on the refined inequality $\left\lfloor\frac{a}{b}\right\rfloor \geq \frac{a + 1}{b} - 1$, which holds for any positive integers $a , b$, we have
\begin{align*}
\left\lfloor\frac{n}{p - 1}\right\rfloor - \sum_{i = 1}^{e} \left\lfloor\frac{n}{p^i}\right\rfloor & \leq \frac{n}{p - 1} - \sum_{i = 1}^{e} \left(\frac{n + 1}{p^i} - 1\right) \\[1mm]
& = \frac{n}{p - 1} - \frac{n + 1}{p - 1} \left(1 - \frac{1}{p^e}\right) + e \\[1mm]
& = \frac{1}{p - 1} \left(\frac{n + 1}{p^e} - 1\right) + e .
\end{align*}
But from the definition of $e$, we have $p^{e + 1} > n + 1$, that is $\frac{n + 1}{p^e} < p$. By reporting this into the last estimate, we get
$$
\left\lfloor\frac{n}{p - 1}\right\rfloor - \sum_{i = 1}^{e} \left\lfloor\frac{n}{p^i}\right\rfloor < e + 1 .
$$
Next, since $\lfloor\frac{n}{p - 1}\rfloor - \sum_{i = 1}^{e} \lfloor\frac{n}{p^i}\rfloor \in\Z$, we conclude to
$$
\left\lfloor\frac{n}{p - 1}\right\rfloor - \sum_{i = 1}^{e} \left\lfloor\frac{n}{p^i}\right\rfloor \leq e ,
$$
confirming the right inequality of \eqref{eq8}. This completes this proof.
\end{proof}

From Proposition \ref{p3}, we derive an asymptotic estimate for the number $\log{\sigma_n}$ when $n$ tends to infinity. We have the following

\begin{coll}\label{coll3}
We have
$$
\log\sigma_n ~\sim_{+ \infty}~ n \log{n} .
$$
\end{coll}

\begin{proof}
According to Proposition \ref{p3}, we have for any $n \in \N^*$:
$$
\log{(n + 1)!} \leq \log{\sigma_n} \leq \log{(n!)} + \log\lcm(1 , 2 , \dots , n , n + 1) .
$$
Then the asymptotic estimate of the corollary follows from the facts: $\log{(n + 1)!}$ $\sim_{+ \infty} \log(n!) \sim_{+ \infty} n \log{n}$ (according to Stirling's formula) and $\log\lcm(1 , 2 , \dots , n ,$ \linebreak $n + 1) \sim_{+ \infty} n$ (according to the prime number theorem).
\end{proof}

Note that the asymptotic estimate of the above corollary will be specified in §\ref{sec4}.

We now turn to establish a result evaluating the $p$-adic valuations of the positive integers $\frac{\sigma_n}{n!}$ ($n \in \N^*$) for sufficiently large prime numbers. We discover as a remarkable phenomenon that primes of a special type play a vital role. We find again this phenomenon in §\ref{sec4} when estimating asymptotically $\log{\sigma_n}$. We have the following theorem:

\begin{thm}\label{t2}
Let $n$ be a positive integer and $p$ be a prime number such that:
$$
\sqrt{n + 1} < p \leq n + 1 .
$$
Then, we have
$$
\vartheta_p\left(\frac{\sigma_n}{n!}\right) \in \{0 , 1\} .
$$
Besides, the equality $\vartheta_p\left(\frac{\sigma_n}{n!}\right) = 1$ holds if and only if $p$ has the form
$$
p = \left\lfloor\frac{n}{k} + 1\right\rfloor ,
$$
with $k \in \N^*$ and $k < \sqrt{n + 1} + 1$.
\end{thm}

\begin{proof}
By the definition of $\sigma_n$ and the Legendre formula \eqref{eq25}, we have that
\begin{align}
\vartheta_p\left(\frac{\sigma_n}{n!}\right) & = \vartheta_p\left(\sigma_n\right) - \vartheta_p\left(n!\right) \notag \\[1mm]
& = \left\lfloor\frac{n}{p - 1}\right\rfloor - \frac{n - S_p(n)}{p - 1} \notag \\[1mm]
& = \left\lfloor\frac{n}{p - 1} - \frac{n - S_p(n)}{p - 1}\right\rfloor ~~~~~~~~~~ \left(\text{since } \frac{n - S_p(n)}{p - 1} = \vartheta_p(n!) \in \Z\right) \notag \\[1mm]
& = \left\lfloor\frac{S_p(n)}{p - 1}\right\rfloor . \label{eq9}
\end{align}
The first part of the theorem is then equivalent to the fact $\left\lfloor\frac{S_p(n)}{p - 1}\right\rfloor \in \{0 , 1\}$. So, let us prove this last fact. The hypothesis on $p$ insures that $n < p^2 - 1$, which implies that the representation of the positive integer $n$ in base $p$ has the form $n = \overline{a_1 a_0}_{(p)}$, with $a_0 , a_1 \in \{0 , 1 , \dots , p - 1\}$ and $(a_0 , a_1) \neq (p - 1 , p - 1)$. Consequently, we have $S_p(n) = a_0 + a_1 < 2 (p - 1)$, implying that $\frac{S_p(n)}{p - 1} < 2$; hence $\left\lfloor\frac{S_p(n)}{p - 1}\right\rfloor \in \{0 , 1\}$, as required. This achieves the proof of the first part of the theorem. Now, let us prove the second part of the theorem. \\
\textbullet{} Suppose that $\vartheta_p\left(\frac{\sigma_n}{n!}\right) = 1$ and let us show the existence of $k \in \N^*$, with $k < \sqrt{n + 1} + 1$ such that $p = \left\lfloor\frac{n}{k} + 1\right\rfloor$. As seen above, the representation of $n$ in base $p$ has the form $n = \overline{a_1 a_0}_{(p)} = a_0 + p a_1$, where $a_0 , a_1 \in \{0 , 1 , \dots , p - 1\}$ and $(a_0 , a_1) \neq (p - 1 , p - 1)$. We will show that $k = a_1 + 1$ is suitable. According to \eqref{eq9}, we have $\vartheta_p\left(\frac{\sigma_n}{n!}\right) = \left\lfloor\frac{S_p(n)}{p - 1}\right\rfloor = \left\lfloor\frac{a_0 + a_1}{p - 1}\right\rfloor$. So the supposition $\vartheta_p\left(\frac{\sigma_n}{n!}\right) = 1$ implies that $\frac{a_0 + a_1}{p - 1} \geq 1$, that is $a_0 + a_1 \geq p - 1$. This last inequality together with $a_0 < p$ imply that
$$
p - 1 \leq \frac{a_0 + a_1 p}{a_1 + 1} < p ,
$$
which is equivalent to
$$
\left\lfloor\frac{n}{a_1 + 1}\right\rfloor = p - 1 .
$$
Thus
$$
p = \left\lfloor\frac{n}{a_1 + 1} + 1\right\rfloor .
$$
Besides, we have $a_1 = \left\lfloor\frac{n}{p}\right\rfloor \leq \frac{n}{p} < \sqrt{n + 1}$ (since $p > \sqrt{n + 1} > \frac{n}{\sqrt{n + 1}}$). Thus $k = a_1 + 1$ satisfy the required properties (i.e., $p = \left\lfloor\frac{n}{k} + 1\right\rfloor$ and $k < \sqrt{n + 1} + 1$). \\
\textbullet{} Conversely, suppose that there exists $k \in \N^*$, with $k < \sqrt{n + 1} + 1$, such that $p = \left\lfloor\frac{n}{k} + 1\right\rfloor$, and let us show that $\vartheta_p\left(\frac{\sigma_n}{n!}\right) = 1$. Setting $a_0 := n - (k - 1) p$ and $a_1 := k - 1$, we first show that the representation of $n$ in base $p$ is $n = \overline{a_1 a_0}_{(p)}$. Since it is immediate that $n = a_0 + p a_1$, it just remains to prove that $a_0 , a_1 \in \{0 , 1 , \dots , p - 1\}$. Since $k < \sqrt{n + 1} + 1 < p + 1$ then $k - 1 < p$; that is $a_1 \in \{0 , 1 , \dots , p - 1\}$. Next, since $p = \left\lfloor\frac{n}{k} + 1\right\rfloor$ then
$$
p \leq \frac{n}{k} + 1 < p + 1 ,
$$
implying that
$$
p - k \leq n - (k - 1) p < p ,
$$
that is
$$
p - k \leq a_0 < p .
$$
But $p - k = (p - 1) - a_1 \geq 0$; thus $a_0 \in \{0 , 1 , \dots , p - 1\}$. We have confirmed that the representation of $n$ in base $p$ is $n = \overline{a_1 a_0}_{(p)}$. Consequently, we have (according to \eqref{eq9}):
$$
\vartheta_p\left(\frac{\sigma_n}{n!}\right) = \left\lfloor\frac{S_p(n)}{p - 1}\right\rfloor = \left\lfloor\frac{a_0 + a_1}{p - 1}\right\rfloor = \left\lfloor\frac{n - (k - 1)(p - 1)}{p - 1}\right\rfloor = \left\lfloor\frac{n}{p - 1}\right\rfloor - k + 1 .
$$
Then, since $\frac{n}{p - 1} \geq k$ (because $\frac{n}{k} + 1 \geq \left\lfloor\frac{n}{k} + 1\right\rfloor = p$), it follows that $\vartheta_p\left(\frac{\sigma_n}{n!}\right) \geq 1$. But since $\vartheta_p\left(\frac{\sigma_n}{n!}\right) \in \{0 , 1\}$ (according to the first part, already proved, of the theorem), we conclude that $\vartheta_p\left(\frac{\sigma_n}{n!}\right) = 1$, as required. This completes the proof of the theorem.
\end{proof}

\section[Analytic estimates]{Analytic estimates of the numbers $\log{\rho_n}$ and $\log{\sigma_n}$}\label{sec4}
Throughout this section, we let $\c$ denote the absolute positive constant given by:
$$
\c := \sum_{p \text{ prime}} \frac{\log{p}}{p (p - 1)} = 0.755\dots .
$$
Our goal is to find asymptotic estimates for $\log{\rho_n}$ and $\log{\sigma_n}$ as $n$ tends to infinity. The obtained main results are the following:
\begin{thm}\label{t3}
We have
$$
\log{\rho_n} = n \log{n} - (\c + 1) n + O\left(\sqrt{n}\right) .
$$
\end{thm}
\begin{thm}\label{t4}
We have
$$
\log{\sigma_n} = n \log{n} - n + O\left(\sqrt{n \log{n}}\right) .
$$
\end{thm}
\begin{conj}[improving Theorem \ref{t4}]\label{conj1}
We have
$$
\log{\sigma_n} = n \log{n} - n + O\left(\sqrt{n}\right) .
$$
\end{conj}

Note that an explanation for the validity of Conjecture \ref{conj1} is given latter; actually, it depends on a conjecture on counting the prime numbers of a certain form, which is heuristically plausible. To establish the above results, we need the following auxiliary results: 

\begin{lemma}\label{l2}
For any $x \geq 1$, we have
$$
\sum_{\begin{subarray}{c}
p \text{ prime} \\
p > x
\end{subarray}} \frac{\log{p}}{p (p - 1)} = O\left(\frac{1}{x}\right) .
$$
\end{lemma}
\begin{proof}
Since $\frac{\log{p}}{p (p - 1)} \leq 2 \frac{\log{p}}{p^2}$ (for any prime number $p$), then it suffices to show that $\sum_{p \text{ prime, } p > x} \frac{\log{p}}{p^2} = O\left(\frac{1}{x}\right)$. According to the Abel summation formula (see e.g., \cite[Proposition 1.4]{kon}), we have for any positive real numbers $x , y$, with $x < y$:
\begin{align*}
\sum_{\begin{subarray}{c}
p \text{ prime} \\
x < p \leq y
\end{subarray}} \frac{\log{p}}{p^2} & = \left(\sum_{\begin{subarray}{c}
p \text{ prime} \\
x < p \leq y
\end{subarray}} \log{p}\right) \frac{1}{y^2} - \bigints_{x}^{y} \left(\sum_{\begin{subarray}{c}
p \text{ prime} \\
x < p \leq t
\end{subarray}} \log{p}\right) \left(\frac{1}{t^2}\right)' \, d t \\[2mm]
& = \frac{\theta(y) - \theta(x)}{y^2} + 2 \int_{x}^{y} \frac{\theta(t) - \theta(x)}{t^3} \, d t .
\end{align*}
Then, by setting $y$ to infinity, it follows (since $\theta(y) = O(y)$) that:
$$
\sum_{\begin{subarray}{c}
p \text{ prime} \\
p > x
\end{subarray}} \frac{\log{p}}{p^2} = 2 \int_{x}^{+ \infty} \frac{\theta(t) - \theta(x)}{t^3} \, d t = 2 \int_{x}^{+ \infty} \frac{\theta(t)}{t^3} \, d t - \frac{\theta(x)}{x^2} .
$$
Using finally $\theta(t) = O(t)$, we get
$$
\sum_{\begin{subarray}{c}
p \text{ prime} \\
p > x
\end{subarray}} \frac{\log{p}}{p^2} = O\left(\int_{x}^{+ \infty} \frac{d t}{t^2}\right) + O\left(\frac{1}{x}\right) = O\left(\frac{1}{x}\right) ,
$$
as required. The proof is complete.
\end{proof}

Lemma \ref{l2} above is used in the proof of the following proposition:

\begin{prop}\label{p10}
For any positive integer $n$, we have
$$
\sum_{p \text{ prime}} \left(\left\lfloor\frac{n}{p^2}\right\rfloor + \left\lfloor\frac{n}{p^3}\right\rfloor + \dots\right) \log{p} = \c \cdot n + O\left(\sqrt{n}\right) .
$$
\end{prop}
\begin{proof}
Let $n$ be a fixed positive integer. For any prime number $p$, let $e_p$ denote the greatest natural number satisfying $p^{e_p} \leq n$; explicitly $e_p = \lfloor\frac{\log{n}}{\log{p}}\rfloor$. So we have $p^{e_p + 1} > n$. On the one hand, we have
\begin{align*}
\sum_{p \text{ prime}} \left(\left\lfloor\frac{n}{p^2}\right\rfloor + \left\lfloor\frac{n}{p^3}\right\rfloor + \dots\right) \log{p} & \leq \sum_{p \text{ prime}} \left(\frac{n}{p^2} + \frac{n}{p^3} + \dots\right) \log{p} \\
& = \sum_{p \text{ prime}} \frac{n}{p (p - 1)} \log{p} ; 
\end{align*}
that is
\begin{equation}\label{eq10}
\sum_{p \text{ prime}} \left(\left\lfloor\frac{n}{p^2}\right\rfloor + \left\lfloor\frac{n}{p^3}\right\rfloor + \dots\right) \log{p} \leq \c \cdot n .
\end{equation}
On the other hand, we have (according to the definition of the $e_p$'s):
\begin{align*}
\sum_{p \text{ prime}} \left(\left\lfloor\frac{n}{p^2}\right\rfloor + \left\lfloor\frac{n}{p^3}\right\rfloor + \dots\right) \log{p} & = \sum_{\begin{subarray}{c}
p \text{ prime} \\
p \leq \sqrt{n}
\end{subarray}} \left(\left\lfloor\frac{n}{p^2}\right\rfloor + \left\lfloor\frac{n}{p^3}\right\rfloor + \dots + \left\lfloor\frac{n}{p^{e_p}}\right\rfloor\right) \log{p} \\
& \hspace*{-5cm} \geq \sum_{\begin{subarray}{c}
p \text{ prime} \\
p \leq \sqrt{n}
\end{subarray}} \left[\left(\frac{n}{p^2} - 1\right) + \left(\frac{n}{p^3} - 1\right) + \dots + \left(\frac{n}{p^{e_p}} - 1\right)\right] \log{p} \\
& \hspace*{-5cm} = n \sum_{\begin{subarray}{c}
p \text{ prime} \\
p \leq \sqrt{n}
\end{subarray}}
\left(\frac{1}{p^2} + \frac{1}{p^3} + \dots + \frac{1}{p^{e_p}}\right) \log{p} - \sum_{\begin{subarray}{c}
p \text{ prime} \\
p \leq \sqrt{n}
\end{subarray}} \left(e_p - 1\right) \log{p} \\
& \hspace*{-5cm} = n \sum_{\begin{subarray}{c}
p \text{ prime} \\
p \leq \sqrt{n}
\end{subarray}} \left(\frac{1}{p (p - 1)} - \frac{1}{p^{e_p} (p - 1)}\right) \log{p} - \sum_{\begin{subarray}{c}
p \text{ prime} \\
p \leq \sqrt{n}
\end{subarray}} \left(e_p - 1\right) \log{p} \\
& \hspace*{-5cm} = n \sum_{\begin{subarray}{c}
p \text{ prime} \\
p \leq \sqrt{n}
\end{subarray}} \frac{\log{p}}{p (p - 1)} - n \sum_{\begin{subarray}{c}
p \text{ prime} \\
p \leq \sqrt{n}
\end{subarray}} \frac{\log{p}}{p^{e_p} (p - 1)} - \sum_{\begin{subarray}{c}
p \text{ prime} \\
p \leq \sqrt{n}
\end{subarray}} \left(e_p - 1\right) \log{p} \\
& \hspace*{-5cm} = n \left(\c - \sum_{\begin{subarray}{c}
p \text{ prime} \\
p > \sqrt{n}
\end{subarray}} \frac{\log{p}}{p (p - 1)}\right) - n \sum_{\begin{subarray}{c}
p \text{ prime} \\
p \leq \sqrt{n}
\end{subarray}} \frac{\log{p}}{p^{e_p} (p - 1)} - \sum_{\begin{subarray}{c}
p \text{ prime} \\
p \leq \sqrt{n}
\end{subarray}} \left(e_p - 1\right) \log{p} ;  
\end{align*}
that is
\begin{multline}\label{eq28}
\sum_{p \text{ prime}} \left(\left\lfloor\frac{n}{p^2}\right\rfloor + \left\lfloor\frac{n}{p^3}\right\rfloor + \dots\right) \log{p} \geq \c \, n - n \sum_{\begin{subarray}{c}
p \text{ prime} \\
p > \sqrt{n}
\end{subarray}} \frac{\log{p}}{p (p - 1)} - n \sum_{\begin{subarray}{c}
p \text{ prime} \\
p \leq \sqrt{n}
\end{subarray}} \frac{\log{p}}{p^{e_p} (p - 1)} \\
- \sum_{\begin{subarray}{c}
p \text{ prime} \\
p \leq \sqrt{n}
\end{subarray}} \left(e_p - 1\right) \log{p} .
\end{multline}
But, by using Lemma \ref{l2}, we have
\begin{equation}\label{eq29}
\sum_{\begin{subarray}{c}
p \text{ prime} \\
p > \sqrt{n}
\end{subarray}} \frac{\log{p}}{p (p - 1)} = O\left(\frac{1}{\sqrt{n}}\right) .
\end{equation}
Next, by using the fact $p^{e_p} > \frac{n}{p}$ (for any prime $p$), we have
\begin{equation}\label{eq30}
\sum_{\begin{subarray}{c}
p \text{ prime} \\
p \leq \sqrt{n}
\end{subarray}} \frac{\log{p}}{p^{e_p} (p - 1)} < \frac{1}{n} \sum_{\begin{subarray}{c}
p \text{ prime} \\
p \leq \sqrt{n}
\end{subarray}} \frac{p}{p - 1} \log{p} \leq \frac{2}{n} \sum_{\begin{subarray}{c}
p \text{ prime} \\
p \leq \sqrt{n}
\end{subarray}} \log{p} = \frac{2}{n} \theta\left(\sqrt{n}\right) = O\left(\frac{1}{\sqrt{n}}\right) ,
\end{equation}
and by using the fact $e_p - 1 < e_p := \lfloor\frac{\log{n}}{\log{p}}\rfloor \leq \frac{\log{n}}{\log{p}}$, we have
\begin{equation}\label{eq31}
\sum_{\begin{subarray}{c}
p \text{ prime} \\
p \leq \sqrt{n}
\end{subarray}} \left(e_p - 1\right) \log{p} < \sum_{\begin{subarray}{c}
p \text{ prime} \\
p \leq \sqrt{n}
\end{subarray}} \log{n} = (\log{n}) \pi(\sqrt{n}) = O\left(\sqrt{n}\right) .
\end{equation}
Then, by inserting \eqref{eq29}, \eqref{eq30} and \eqref{eq31} into \eqref{eq28}, we get
\begin{equation}\label{eq13}
\sum_{p \text{ prime}} \left(\left\lfloor\frac{n}{p^2}\right\rfloor + \left\lfloor\frac{n}{p^3}\right\rfloor + \dots\right) \log{p} \geq \c \, n + O\left(\sqrt{n}\right) .
\end{equation}
Finally, \eqref{eq10} and \eqref{eq13} conclude to
$$
\sum_{p \text{ prime}} \left(\left\lfloor\frac{n}{p^2}\right\rfloor + \left\lfloor\frac{n}{p^3}\right\rfloor + \dots\right) \log{p} = \c \cdot n + O\left(\sqrt{n}\right) ,
$$
as required.
\end{proof}

We are now able to prove Theorem \ref{t3}.

\begin{proof}[Proof of Theorem \ref{t3}]
For any sufficiently large integer $n$, we have according to Legendre's formula:
\begin{align*}
\log{\rho_n} = \sum_{p \text{ prime}} \left\lfloor\frac{n}{p}\right\rfloor \log{p} & = \sum_{p \text{ prime}}\left(\left\lfloor\frac{n}{p}\right\rfloor + \left\lfloor\frac{n}{p^2}\right\rfloor + \dots\right) \log{p} \\
& \hspace*{3cm} - \sum_{p \text{ prime}}\left(\left\lfloor\frac{n}{p^2}\right\rfloor + \left\lfloor\frac{n}{p^3}\right\rfloor + \dots\right) \log{p} \\[2mm]
& = \log(n!) - \sum_{p \text{ prime}}\left(\left\lfloor\frac{n}{p^2}\right\rfloor + \left\lfloor\frac{n}{p^3}\right\rfloor + \dots\right) \log{p} .
\end{align*}
Then, the weaker form of Stirling's approximation formula $\log(n!) = n \log{n} - n + O(\log{n})$ and Proposition \ref{p10} conclude to:
$$
\log{\rho_n} = n \log{n} - (\c + 1) n + O(\sqrt{n}) ,
$$
as required.
\end{proof}

We now turn to estimate $\log{\sigma_n}$ by leaning on the estimate of $\log{\rho_n}$ (given by Theorem \ref{t3} proved above). To do so, we shall first establish an important formula relating $\log{\rho_n}$ and $\log{\sigma_n}$. This is done with the following proposition:

\begin{prop}\label{p4}
For any positive integer $n$, we have
\begin{align}
\log{\rho_n} & = \sum_{k = 1}^{n} \theta\left(\frac{n}{k}\right) , \label{eq14} \\
\log{\sigma_n} & = \sum_{k = 1}^{n} \theta\left(\frac{n}{k} + 1\right) , \label{eq15} \\
\log{\sigma_n} - \log{\rho_n} & = \sum_{\begin{subarray}{c}
1 \leq k \leq n \\
\lfloor\frac{n}{k} + 1\rfloor \text{ is prime}
\end{subarray}} \!\!\!\!\!\!\log\left\lfloor\frac{n}{k} + 1\right\rfloor . \label{eq16}
\end{align}
\end{prop}

\begin{proof}
Let $n$ be a fixed positive integer. We have
\begin{multline*}
\log{\rho_n} = \sum_{p \text{ prime}} \left\lfloor\frac{n}{p}\right\rfloor \log{p} = \sum_{p \text{ prime}} \left(\sum_{1 \leq k \leq \frac{n}{p}} 1\right) \log{p} = \sum_{1 \leq k \leq n} \left(\sum_{\begin{subarray}{c}
p \text{ prime} \\
p \leq \frac{n}{k}
\end{subarray}} \log{p}\right) \\
= \sum_{1 \leq k \leq n} \theta\left(\frac{n}{k}\right) ,
\end{multline*}
proving \eqref{eq14}. Similarly, we have
\begin{multline*}
\log{\sigma_n} = \sum_{p \text{ prime}} \left\lfloor\frac{n}{p - 1}\right\rfloor \log{p} = \sum_{p \text{ prime}} \left(\sum_{1 \leq k \leq \frac{n}{p - 1}} 1\right) \log{p} = \sum_{1 \leq k \leq n} \left(\sum_{\begin{subarray}{c}
p \text{ prime} \\
p \leq \frac{n}{k} + 1
\end{subarray}} \log{p}\right) \\[1mm]
= \sum_{1 \leq k \leq n} \theta\left(\frac{n}{k} + 1\right) ,
\end{multline*}
proving \eqref{eq15}. Finally, using \eqref{eq14} and \eqref{eq15}, let us prove \eqref{eq16}. We have
\begin{align*}
\log{\sigma_n} - \log{\rho_n} & = \sum_{1 \leq k \leq n} \theta\left(\frac{n}{k} + 1\right) - \sum_{1 \leq k \leq n} \theta\left(\frac{n}{k}\right) \\
& = \sum_{1 \leq k \leq n} \left(\theta\left(\frac{n}{k} + 1\right) - \theta\left(\frac{n}{k}\right)\right) \\
& = \sum_{1 \leq k \leq n} \left(\sum_{\begin{subarray}{c}
p \text{ prime} \\
\frac{n}{k} < p \leq \frac{n}{k} + 1
\end{subarray}} \log{p}\right) .
\end{align*}
But since for any $1 \leq k \leq n$, the interval $(\frac{n}{k} , \frac{n}{k} + 1]$ contains a unique integer which is $\lfloor\frac{n}{k} + 1\rfloor$, it follows that:
$$
\log{\sigma_n} - \log{\rho_n} = \sum_{\begin{subarray}{c}
1 \leq k \leq n \\
\lfloor\frac{n}{k} + 1\rfloor \text{ is prime}
\end{subarray}} \!\!\!\!\!\!\log\left\lfloor\frac{n}{k} + 1\right\rfloor ,
$$
proving \eqref{eq16}. The proof is complete.
\end{proof}

To deduce an asymptotic estimate for $\log{\sigma_n}$ from that of $\log{\rho_n}$ by means of Proposition \ref{p4}, we must estimate (asymptotically) the sum
$$
\mathcal{S}(n) := \sum_{\begin{subarray}{c}
1 \leq k \leq n \\
\lfloor\frac{n}{k} + 1\rfloor \text{ is prime}
\end{subarray}} \!\!\!\!\!\!\log\left\lfloor\frac{n}{k} + 1\right\rfloor .
$$
To do so, we split $\mathcal{S}$ into two sums:
$$
\mathcal{S}_1(n) := \sum_{\begin{subarray}{c}
1 \leq k \leq \sqrt{n} \\
\lfloor\frac{n}{k} + 1\rfloor \text{ is prime}
\end{subarray}} \!\!\!\!\!\!\log\left\lfloor\frac{n}{k} + 1\right\rfloor ~~~~\text{and}~~~~ \mathcal{S}_2(n) := \sum_{\begin{subarray}{c}
\sqrt{n} < k \leq n \\
\lfloor\frac{n}{k} + 1\rfloor \text{ is prime}
\end{subarray}} \!\!\!\!\!\!\log\left\lfloor\frac{n}{k} + 1\right\rfloor .
$$
The estimate of $\mathcal{S}_2(n)$ is within our reach; it is given by the following proposition:
\begin{prop}\label{p5}
For any positive integer $n$, we have
$$
\mathcal{S}_2(n) = \c \, n + O\left(\sqrt{n}\right) .
$$
\end{prop}

\begin{proof}
For a given $n \in \N^*$, we have
$$
\mathcal{S}_2(n) := \sum_{\begin{subarray}{c}
\sqrt{n} < k \leq n \\
\lfloor\frac{n}{k} + 1\rfloor \text{ is prime}
\end{subarray}} \!\!\!\!\!\!\log\left\lfloor\frac{n}{k} + 1\right\rfloor = \sum_{\begin{subarray}{c}
p \text{ prime} \\
p < \sqrt{n} + 1
\end{subarray}} \sum_{\begin{subarray}{c}
\sqrt{n} < k \leq n \\
\left\lfloor\frac{n}{k} + 1\right\rfloor = p
\end{subarray}} \!\!\!\! \log{p} .
$$
But since for any prime number $p$ satisfying $p < \sqrt{n} + 1$ and any integer $k$ satisfying $\sqrt{n} < k \leq n$, we have
$$
\left\lfloor\frac{n}{k} + 1\right\rfloor = p \Longleftrightarrow p \leq \frac{n}{k} + 1 < p + 1 \Longleftrightarrow \frac{n}{p} < k \leq \frac{n}{p - 1} ,
$$
it follows that:
\begin{equation}\label{eq17}
\mathcal{S}_2(n) = \sum_{\begin{subarray}{c}
p \text{ prime} \\
p < \sqrt{n} + 1
\end{subarray}} \sum_{\begin{subarray}{c}
\sqrt{n} < k \leq n \\
\frac{n}{p} < k \leq \frac{n}{p - 1}
\end{subarray}} \!\!\!\! \log{p} = \sum_{\begin{subarray}{c}
p \text{ prime} \\
p < \sqrt{n} + 1
\end{subarray}} \!\!\left(\sum_{\max\left(\sqrt{n} , \frac{n}{p}\right) < k \leq \frac{n}{p - 1}} 1\right) \log{p} .
\end{equation}
Next, we remark that for any prime number $p$ satisfying $p < \sqrt{n} + 1$, we have $\frac{n}{p} > \frac{n}{\sqrt{n} + 1} > \sqrt{n} - 1$, implying that $\frac{n}{p} \leq \max\left(\sqrt{n} , \frac{n}{p}\right) < \frac{n}{p} + 1$. Consequently, the interval $\left(\max\left(\sqrt{n} , \frac{n}{p}\right) , \frac{n}{p - 1}\right]$ contains at least \\
$\frac{n}{p - 1} - \max\left(\sqrt{n} , \frac{n}{p}\right) - 1 > \frac{n}{p - 1} - \frac{n}{p} - 2 = \frac{n}{p (p - 1)} - 2$ integers and at most \\
$\frac{n}{p - 1} - \max\left(\sqrt{n} , \frac{n}{p}\right) + 1 \leq \frac{n}{p - 1} - \frac{n}{p} + 1 = \frac{n}{p(p - 1)} + 1$ integers. So, for any prime number $p$ satisfying $p < \sqrt{n} + 1$, we have
$$
\sum_{\max\left(\sqrt{n} , \frac{n}{p}\right) < k \leq \frac{n}{p - 1}} 1 = \frac{n}{p (p - 1)} + O_{\perp p}(1) . 
$$
By reporting this into \eqref{eq17}, we get
\begin{align*}
\mathcal{S}_2(n) & = \sum_{\begin{subarray}{c}
p \text{ prime} \\
p < \sqrt{n} + 1
\end{subarray}} \left(\frac{n}{p (p - 1)} + O_{\perp p}(1)\right) \log{p} \\
& = \left(\sum_{\begin{subarray}{c}
p \text{ prime} \\
p < \sqrt{n} + 1
\end{subarray}} \frac{\log{p}}{p (p - 1)}\right) n + O\left(\theta(\sqrt{n} + 1)\right) \\
& = \left(\c - \sum_{\begin{subarray}{c}
p \text{ prime} \\
p \geq \sqrt{n} + 1
\end{subarray}} \frac{\log{p}}{p (p - 1)}\right) n + O\left(\theta(\sqrt{n} + 1)\right) .
\end{align*}
But since
$$
\sum_{\begin{subarray}{c}
p \text{ prime} \\
p \geq \sqrt{n} + 1
\end{subarray}} \frac{\log{p}}{p (p - 1)} \leq 2 \!\!\!\sum_{\begin{subarray}{c}
p \text{ prime} \\
p \geq \sqrt{n} + 1
\end{subarray}} \frac{\log{p}}{p^2} = O\left(\frac{1}{\sqrt{n}}\right) ~~~~~~~ (\text{according to Lemma \ref{l2}})
$$
and
$$
\theta\left(\sqrt{n} + 1\right) = O\left(\sqrt{n}\right) ~~~~~~~~~~ (\text{according to the Chebyshev estimates}) ,
$$
we conclude to
$$
\mathcal{S}_2(n) = \c \, n + O\left(\sqrt{n}\right) ,
$$
as required.
\end{proof}

We now turn to estimate the crucial sum $\mathcal{S}_1(n)$. To facilitate the task, we begin by estimating $\mathcal{S}_1(n)$ in terms of the cardinality of a specific set of prime numbers. For any $n \in \N^*$, we set
$$
\mathscr{A}_n := \left\{\left\lfloor\frac{n}{k} + 1\right\rfloor ~;~ k \in \N^* , k \leq \sqrt{n}\right\} \cap \mathscr{P}
$$
(in other words, $\mathscr{A}_n$ is the set of prime numbers having the form $\left\lfloor\frac{n}{k} + 1\right\rfloor$, where $k \leq \sqrt{n}$ is a positive integer). Above all, it is important to note that for any $n \in \N^*$, the positive integers $\left\lfloor\frac{n}{k} + 1\right\rfloor$ ($k \in \N^*$, $k \leq \sqrt{n}$) (appearing in the definition of $\mathscr{A}_n$) are pairwise distinct. Indeed, for all $k , \ell \in \N^*$, with $k \leq \sqrt{n}$, $\ell \leq \sqrt{n}$, and $k \neq \ell$, we have
$$
\left\vert\left(\frac{n}{k} + 1\right) - \left(\frac{n}{\ell} + 1\right)\right\vert = \frac{n \vert \ell - k\vert}{k \ell} \geq \frac{n}{k \ell} \geq 1 ,
$$
implying that:
$$
\left\lfloor\frac{n}{k} + 1\right\rfloor \neq \left\lfloor\frac{n}{\ell} + 1\right\rfloor .
$$
It follows from this fact that $\mathscr{A}_n$ ($n \in \N^*$) has the same cardinality with the set of positive integers $k$ satisfying $k \leq \sqrt{n}$ and for which $\left\lfloor\frac{n}{k} + 1\right\rfloor$ is prime. That is
\begin{equation}\label{eq18}
\card \mathscr{A}_n = \sum_{\begin{subarray}{c}
1 \leq k \leq \sqrt{n} \\
\left\lfloor\frac{n}{k} + 1\right\rfloor \text{ is prime}
\end{subarray}} 1 .
\end{equation}
The estimate of $\mathcal{S}_1(n)$ in terms of $\card \mathscr{A}_n$ ($n \in \N^*$) is given by the following proposition:
\begin{prop}\label{p6}
We have
$$
\mathcal{S}_1(n) = \left(\card \mathscr{A}_n\right) \cdot O\left(\log{n}\right) .
$$
\end{prop}
\begin{proof}
Let $n \in \N^*$ be fixed. From the obvious double inequality
$$
\sqrt{n} < \left\lfloor\frac{n}{k} + 1\right\rfloor \leq n + 1 ~~~~~~~~~~ (\forall k \in \N^* ,\text{ with } k \leq \sqrt{n}) ,
$$
we have that
$$
\frac{1}{2} \log(n) \cdot \!\!\!\!\!\!\sum_{\begin{subarray}{c}
1 \leq k \leq \sqrt{n} \\
\left\lfloor\frac{n}{k} + 1\right\rfloor \text{ is prime}
\end{subarray}} \!\!\!\!\!\! 1 \leq \mathcal{S}_1(n) \leq \log(n + 1) \cdot \!\!\!\!\!\! \sum_{\begin{subarray}{c}
1 \leq k \leq \sqrt{n} \\
\left\lfloor\frac{n}{k} + 1\right\rfloor \text{ is prime}
\end{subarray}} \!\!\!\!\!\! 1 ;
$$
which immediately implies (taking into account \eqref{eq18}):
$$
\mathcal{S}_1(n) = \left(\card \mathscr{A}_n\right) \cdot O\left(\log{n}\right) ,
$$
as required.
\end{proof}

By combining Formula \eqref{eq16} of Proposition \ref{p4}, Theorem \ref{t3}, Proposition \ref{p5}, and Proposition \ref{p6}, we immediately derive the following proposition:

\begin{prop}\label{p7}
We have
\begin{equation}
\log{\sigma_n} = n \log{n} - n + O\left(\sqrt{n}\right) + \left(\card \mathscr{A}_n\right) \cdot O\left(\log{n}\right) . \tag*{$\square$}
\end{equation}
\end{prop}

At this point, the whole problem is now to estimate $\card \mathscr{A}_n$ ($n \in \N^*$). First, let us do it heuristically. For $n \in \N^*$, since the numbers constituting the set $\left\{\left\lfloor\frac{n}{k} + 1\right\rfloor ~;~ k \in \N^* , k \leq \sqrt{n}\right\}$ (of cardinality $\lfloor\sqrt{n}\rfloor$) do not satisfy (apparently) any particular congruence, we may conjecture with considerable confidence that the quantity prime numbers in it (i.e., $\card \mathscr{A}_n$) is $O\left(\frac{\sqrt{n}}{\log{\sqrt{n}}}\right) = O\left(\frac{\sqrt{n}}{\log{n}}\right)$. We precisely make the following

\begin{conj}\label{conj2}
There exist two positive absolute constants $\alpha$ and $\beta$ {\rm(}with $\alpha < \beta${\rm)} such that for any sufficiently large positive integer $n$, we have
$$
\alpha \left(\frac{\sqrt{n}}{\log{n}}\right) \leq \card \mathscr{A}_n \leq \beta \left(\frac{\sqrt{n}}{\log{n}}\right) .
$$
\end{conj}

The proof of Conjecture \ref{conj1} through Conjecture \ref{conj2} is then immediate:

\begin{proof}[Proof of Conjecture \ref{conj1} through Conjecture \ref{conj2}]
It suffices to insert the estimate $\card \mathscr{A}_n = O\left(\frac{\sqrt{n}}{\log{n}}\right)$ (provided by Conjecture \ref{conj2}) into the estimate of Proposition \ref{p7}.
\end{proof}

Unfortunately, we were unable to confirm Conjecture \ref{conj2}, so the best result we have achieved is Theorem \ref{t4}. Before setting out the proof of that theorem, it is important to note that the trivial estimate $\card \mathscr{A}_n = O\left(\sqrt{n}\right)$ leads us (through Proposition \ref{p7}) to the estimate:
$$
\log{\sigma_n} = n \log{n} - n + O\left(\sqrt{n} \log{n}\right) .
$$
We shall improve the later by counting more intelligently the elements of $\mathscr{A}_n$. We have the following

\begin{prop}\label{p8}
We have
$$
\card \mathscr{A}_n = O\left(\sqrt{\frac{n}{\log{n}}}\right) .
$$
\end{prop}

\begin{proof}
Let $n$ be a fixed sufficiently large positive integer and let $t$ be a real parameter with $1 \leq t \leq \sqrt{n}$ (we will choose $t$ later in terms of $n$ in order to optimize the result). We have (according to \eqref{eq18})
\begin{align}
\card \mathscr{A}_n & = \sum_{\begin{subarray}{c}
1 \leq k \leq \sqrt{n} \\
\left\lfloor\frac{n}{k} + 1\right\rfloor \text{ is prime}
\end{subarray}} 1 \notag \\[1mm]
& = \sum_{\begin{subarray}{c}
1 \leq k \leq \frac{\sqrt{n}}{t} \\
\left\lfloor\frac{n}{k} + 1\right\rfloor \text{ is prime}
\end{subarray}} 1 + \sum_{\begin{subarray}{c}
\frac{\sqrt{n}}{t} < k \leq \sqrt{n} \\
\left\lfloor\frac{n}{k} + 1\right\rfloor \text{ is prime}
\end{subarray}} 1 . \label{eq19}
\end{align}
Next, we have on the one hand:
\begin{equation}\label{eq20}
\sum_{\begin{subarray}{c}
1 \leq k \leq \frac{\sqrt{n}}{t} \\
\left\lfloor\frac{n}{k} + 1\right\rfloor \text{ is prime}
\end{subarray}} \!\!\!\!\!\! 1 \leq \sum_{1 \leq k \leq \frac{\sqrt{n}}{t}} 1 \leq \frac{\sqrt{n}}{t} ,
\end{equation}
and on the other hand:
\begin{align*}
\sum_{\begin{subarray}{c}
\frac{\sqrt{n}}{t} < k \leq \sqrt{n} \\
\left\lfloor\frac{n}{k} + 1\right\rfloor \text{ is prime}
\end{subarray}} \!\!\!\!\!\! 1 & \leq \card\!\!\left\{p \text{ prime } ~;~ p \leq \sqrt{n} \, t + 1\right\} ~~~~~~ \left(\text{by putting } p = \left\lfloor\frac{n}{k} + 1\right\rfloor\right) \\
& = \pi\left(\sqrt{n} \, t + 1\right) \\
& = O\left(\frac{\sqrt{n} \, t}{\log(\sqrt{n} \, t)}\right) ~~~~~~~~~~ (\text{according to the Chebyshev estimate}) ,
\end{align*}
implying (since $1 \leq t \leq \sqrt{n}$) that:
\begin{equation}\label{eq21}
\sum_{\begin{subarray}{c}
\frac{\sqrt{n}}{t} < k \leq \sqrt{n} \\
\left\lfloor\frac{n}{k} + 1\right\rfloor \text{ is prime}
\end{subarray}} \!\!\!\!\!\! 1 = O\left(\frac{\sqrt{n} \, t}{\log{n}}\right) .
\end{equation}
By inserting \eqref{eq20} and \eqref{eq21} into \eqref{eq19}, we get
$$
\card \mathscr{A}_n = O\left(\frac{\sqrt{n}}{t}\right) + O\left(\frac{\sqrt{n}}{\log{n}} t\right) .
$$
To obtain an optimal result, we must take $t = O\left(\sqrt{\log{n}}\right)$, providing
$$
\card \mathscr{A}_n = O\left(\sqrt{\frac{n}{\log{n}}}\right) ,
$$
as required.
\end{proof}

We are finally ready to prove Theorem \ref{t4}.

\begin{proof}[Proof of Theorem \ref{t4}]
It suffices to insert the estimate of Proposition \ref{p8} into that of Proposition \ref{p7}.
\end{proof}

\section[Concluding remarks]{Concluding remarks about the connection between the arithmetic and the analytic studies}

In this section, we briefly explain how we can derive our asymptotic estimates concerning $\log{\sigma_n}$ (i.e., Theorem \ref{t4} and Conjecture \ref{conj1}) rather from our arithmetic study. For the sequel, we let $\c_1 , \c_2 , \c_3$, etc. denote suitable absolute constants greater than $1$. For any $n \in \N^*$, we can write:
\begin{equation}\label{eq22}
\frac{\sigma_n}{n!} = \prod_{\begin{subarray}{c}
p \text{ prime} \\
p \leq \sqrt{n + 1}
\end{subarray}} p^{\vartheta_p\left(\frac{\sigma_n}{n!}\right)} \cdot \prod_{\begin{subarray}{c}
p \text{ prime} \\
\sqrt{n + 1} < p \leq n + 1
\end{subarray}} p^{\vartheta_p\left(\frac{\sigma_n}{n!}\right)} .
\end{equation}
For the primes $p \leq \sqrt{n + 1}$, we estimate $\vartheta_p\left(\frac{\sigma_n}{n!}\right)$ as follows: let $n = \overline{a_r a_{r - 1} \dots a_0}_{(p)}$ be the representation of $n$ in base $p$ ($r \in \N$, $a_0 , a_1 , \dots , a_r \in \{0 , 1 , \dots , p - 1\}$, and $a_r \neq 0$). Then we have (according to the Legendre formula \eqref{eq25}):
\begin{multline*}
\vartheta_p\left(\frac{\sigma_n}{n!}\right) = \vartheta_p\left(\sigma_n\right) - \vartheta_p(n!) = \left\lfloor\frac{n}{p - 1}\right\rfloor - \frac{n - S_p(n)}{p - 1} \leq \frac{S_p(n)}{p - 1} \leq \frac{(p - 1) (r + 1)}{p - 1} \\
= r + 1 .
\end{multline*}
Thus
$$
p^{\vartheta_p\left(\frac{\sigma_n}{n!}\right)} \leq p^{r + 1} \leq p \, n .
$$
Consequently, we have
\begin{equation}\label{eq23}
\prod_{\begin{subarray}{c}
p \text{ prime} \\
p \leq \sqrt{n + 1}
\end{subarray}} p^{\vartheta_p\left(\frac{\sigma_n}{n!}\right)} \leq \left(\prod_{\begin{subarray}{c}
p \text{ prime} \\
p \leq \sqrt{n + 1}
\end{subarray}} p\right) \cdot n^{\pi(\sqrt{n + 1})} \leq \c_1^{\sqrt{n}}
\end{equation}
(according to the Chebyshev estimates). \\
However, for the primes $p > \sqrt{n + 1}$, we estimate $\vartheta_p\left(\frac{\sigma_n}{n!}\right)$ by using Theorem \ref{t2} saying us that for any such prime $p$, we have $\vartheta_p\left(\frac{\sigma_n}{n!}\right) \in \{0 , 1\}$ and $\vartheta_p\left(\frac{\sigma_n}{n!}\right) = 1$ if and only if $p \in \mathscr{A}'_n$, where
$$
\mathscr{A}'_n := \left\{\left\lfloor\frac{n}{k} + 1\right\rfloor ~;~ k \in \N^* , k < \sqrt{n + 1} + 1\right\} \cap \mathscr{P} .
$$
(Note the resemblance between $\mathscr{A}'_n$ and $\mathscr{A}_n$). So we have
$$
\prod_{\begin{subarray}{c}
p \text{ prime} \\
\sqrt{n + 1} < p \leq n + 1
\end{subarray}} p^{\vartheta_p\left(\frac{\sigma_n}{n!}\right)} = \prod_{p \in \mathscr{A}'_n} p .
$$
By using the results of §\ref{sec4} concerning $\card \mathscr{A}_n$ (almost the same with $\card \mathscr{A}'_n$), we easily derive that the quantity $\prod p^{\vartheta_p\left(\frac{\sigma_n}{n!}\right)}$, where the product is over the primes $p$ satisfying $\sqrt{n + 1} < p \leq n + 1$, is conjecturally bounded below by $\c_2^{\sqrt{n}}$ and bounded above by $\c_3^{\sqrt{n \log{n}}}$ (or conjecturally by $\c_3^{\sqrt{n}}$). By inserting this together with \eqref{eq23} into \eqref{eq22}, we conclude that $\frac{\sigma_n}{n!}$ is conjecturally bounded below by $\c_2^{\sqrt{n}}$ and bounded above by $\c_4^{\sqrt{n \log{n}}}$ (or conjecturally by $\c_4^{\sqrt{n}}$). Notice that this is exactly what obtained in §\ref{sec4}. By this approach, the constant $\c$ (present in the analytic approach when estimating $\log{\rho_n}$ and then eliminated when estimating $\log{\sigma_n}$) does remarkably not appear!

\section*{Acknowledgement}
The authors acknowledge support from the Algerian DGRSDT (Direction G\'en\'erale de la Recherche Scientifique et du D\'eveloppement Technologique).

\rhead{\textcolor{OrangeRed3}{\it References}}

\end{document}